\documentclass[11pt]{article}
\usepackage[totalwidth=13.0cm,totalheight=20.0cm]{geometry}
\usepackage{latexsym,amsthm,amsmath,amssymb,url}
\usepackage[ruled, linesnumbered]{algorithm2e}
\usepackage{mathrsfs}
\usepackage{tikz,authblk}
\usetikzlibrary{decorations.pathreplacing}
\usepackage{tikz}
\usetikzlibrary{arrows.meta}

\newtheorem{theorem}{Theorem}[section]
\newtheorem{lemma}[theorem]{Lemma}

\newtheorem{observation}[theorem]{Observation}
\newtheorem{definition}[theorem]{Definition}
\newtheorem{proposition}[theorem]{Proposition}

\newtheorem{conjecture}[theorem]{Conjecture}
\newtheorem{claim}[theorem]{Claim}
\newtheorem{openProblem}[theorem]{Open Problem}

\begin{document}
	\title{On  the $k$-anti-traceability Conjecture}
	\author{Bin Chen\footnote{Hefei National Laboratory, Hefei, China. {\tt cbfzu03@163.com}}, Stefanie Gerke\footnote{Department of Mathematics, Royal Holloway University of London, Egham, UK. {\tt stefanie.gerke@rhul.ac.uk}}, Gregory Gutin\footnote{Department of Computer Science, Royal Holloway University of London, Egham, UK.  {\tt g.gutin@rhul.ac.uk}},\\ Hui Lei\footnote{School of Statistics and Data Science, Nankai University, Tianjin, China. {\tt hui.lei@rhul.ac.uk}}, Heis Parker-Cox\footnote{Department of Mathematics, Royal Holloway University of London, Egham, UK. {\tt heis.parker-cox.2023@live.rhul.ac.uk}},  Yacong Zhou\footnote{Department of Computer Science, Royal Holloway University of London, Egham, UK.  {\tt yacong.zhou.2021@live.rhul.ac.uk}}}
	\date{}
	\maketitle
	
	\begin{abstract}
		An oriented graph is called $k$-anti-traceable if the subdigraph induced by every subset with $k$ vertices has a hamiltonian anti-directed path. In this paper, we consider an anti-traceability conjecture. In particular we confirm this conjecture holds when $k\leq 4$. We also show that every sufficiently large $k$-anti-traceable oriented graph admits an anti-path that contains $n-o(n)$ vertices.
	\end{abstract}
	
	\section{Introduction}
	
	\noindent
	

One of the fundamental and extensively studied problems in digraph theory is finding sufficient conditions for a digraph to contain a hamiltonian oriented path of a certain kind.  
A well-known result of R\'{e}dei \cite{R} asserts that every {\em tournament} is {\em traceable}, that is, every orientation of the complete graph contains a hamiltonian directed path. Extending R\'{e}dei's result, Chen and Manalastas \cite{CM} proved that every strongly connected digraph with independence number two is traceable. Havet  further strengthened this result in \cite{Ha}.
A digraph $D$ is \emph{$k$-traceable} if its order is at least $k$ and each of its induced subdigraphs of order $k$ is traceable. Note that every $2$-traceable {\em oriented graph} (i.e., a digraph which can obtained by orienting edges of an undirected graph)  is a tournament.  In \cite{v6} van Aardt, Dunbar, Frick, Nielsen and Oellermann  formulated  the following conjecture called the \emph{Traceability Conjecture}\;(TC for short).
	\begin{conjecture} \textnormal{(\cite{v6})}\label{conj}
		For every integer $k\geq2$, every $k$-traceable oriented graph of order at least $2k-1$ is traceable.
	\end{conjecture}
Despite considerable interest, the TC remains unsolved. Improving on previous results   in \cite{v3,v6}   it was shown in  \cite{v4} that every $k$-traceable oriented graph of order at least $6k-20$ is traceable for $k\geq 4$. For $2\leq k \leq 6$ and $k=8$  it is known \cite{v1,v2} that every $k$-traceable oriented graph of order at least $k$ is traceable. For $k=7$ the situation is different as there exists a $7$-traceable graph of order $8$ that is not traceable  (but every $7$-traceable graph of order at least $9$ is traceable) \cite{v2}.

The TC has also been studied  for several classes of digraphs without specific subdigraphs, see, e.g., \cite{v5,Li,CB}.

An oriented path $P=x_1x_2\dots x_p$ in a digraph $D$ is called {\em anti-directed} (or, AT for short) if every two consecutive arcs of $P$ have opposite orientations, i.e., for every $1\le i\le p-2$ we have either $x_ix_{i+1}$ and $x_{i+2}x_{i+1}$ are arcs in $D$ (also denoted by $x_i\to x_{i+1}$ and $x_{i+2}\to x_{i+1}$) or $x_{i+1}x_i$ and $x_{i+1}x_{i+2}$ are arcs in $D$. Note that in an anti-directed path $P$, unlike in a directed path, the first arc of $P$ may be oriented from or to the initial vertex of $P.$ A {\em hamiltonian anti-directed path} of a digraph $D$ includes all vertices of $D$. A digraph is {\em anti-traceable} if it contains a  hamiltonian anti-directed path.

It is well-known that anti-directed paths behave very differently to directed paths. For example not all tournaments have a hamiltonian anti-directed path as we will see shortly. To do so we need the following definitions.
Let $q$ be a prime power with $q\equiv3\;(\text{mod}\;4)$. The \emph{Paley tournament} of order $q$, denoted by $PT_{q}$, is the tournament with vertex set $V(PT_{q})=GF(q)$ and arc set $A(PT_{q})=\{(i,j):\ j-i\;\text{is}\;\text{a}\;\text{nonzero}\;\text{square}\;\text{in}\;GF(q)\}.$ Let $RT_{n}$ be the \emph{rotational tournament} of order $n=2k+1$, that is, $V(RT_{n})=[2k+1]=\{1,2,\dots ,2k+1\}$ and $A(RT_{n})=\{(i,j):\ j=i+t \mbox{ mod } 2k+1, 1\le t\le k\},$ where $(i+t \mbox{ mod } 2k+1)\in [2k+1]$.
	Gr\"{u}nbaum \cite{Gr} proved the following:
	
	\begin{theorem} \textnormal{(\cite{Gr})}\label{tournament}
		Every tournament is anti-traceable unless it is isomorphic to one of the tournaments in $\{PT_{3},RT_{5},PT_{7}\}$.
	\end{theorem}
	
	In \cite{Ro}, Rosenfeld strengthened Theorem~\ref{tournament} by showing that for every tournament $T_n$ of order $n\geq 9$ and every vertex $v$ of $T_{n}$, there is a hamiltonian anti-directed path in $T_{n}$ starting at $v$. The reader is referred to \cite{H-T,Fo,Tho} for more related interesting results with respect to hamiltonian oriented paths in tournaments.
	
A digraph $D$ is \emph{$k$-anti-traceable} (or, $k$-AT for short) if its order is at least $k$ and each of its induced subdigraphs of order $k$ is anti-traceable. 
Not surprisingly $k$-anti-tracable oriented graphs also behave very differently to $k$-traceable graphs. For example, it was shown in \cite{v3}  that all $k$-traceable graphs on $6k-20$ vertices have independence number at most $2$. As the following family of example shows the independence number of a $k$-anti-traceable graph can be as high as $\lceil k/2 \rceil$. Let $n\ge k\geq 16$. We construct an oriented graph $D$ on $n$ vertices by having an independent set $X$ of size $\lceil k/2 \rceil$ and a set $Y$ that forms  a tournament on $n-\lceil k/2 \rceil$ vertices. In addition there is an arc from every vertex in $Y$ to every vertex in $X$. Clearly the independence number of $D$ satisfies $\alpha(D)=|X|=\lceil k/2 \rceil$. We now show that $D$ is $k$-anti-traceable. Consider a subset $K$ of $V(D)$ of size $k$. If  $K\subseteq Y$ then by Theorem~\ref{tournament} it contains an anti-directed path. If $K$ intersects $X$ in $t\geq 1$ vertices then we {\ reserve} $t-1$ vertices in $K \cap Y$. If $k-2t+1$ is even then by Theorem~\ref{tournament} there is an anti-directed path $P=x_1x_2\dots x_{k-2t+1}$ covering all non-reserved vertices  in $K\cap Y$ and moreover, either $x_1\to x_2$ or $x_{k-2t+1} \to x_{k-2t}$. If $x_1\to x_2$, then let $u=x_1$ and $v=x_2$, otherwise, let $u=x_{k-2t+1}$ and 
$v=x_{k-2t}$. Hence we can extend $P$ to an anti-directed path containing all vertices of $K$ by starting with path $P$ in such a direction that the last vertex is $u$ 
and then alternating between the vertices in $X \cap K$ and the reserved vertices. If $k-2t+1$ is odd then chose a non-reserved vertex $a$ in $Y\cap K$ and observe that by Theorem~\ref{tournament}  there is an anti-directed path $Q=x_1x_2\dots x_{k-2t}$ consisting of all non-reserved vertices apart from $a$. Define $u$ and $v$ in $Q$ similarly to the way we  defined them in $P.$ Now, construct an anti-directed path by starting at $a$, using all vertices of $K\cap X$ and all reserved vertices and then using $Q$ starting at $u.$

\vspace{2mm}
	
	Inspired by the above results, we will focus on finding sufficient conditions for more general digraphs to be anti-traceable. 	
	\begin{openProblem}
		For any positive integer $k$, what is the minimum integer $f(k)\geq k$ such that every $k$-anti-traceable oriented graph with order at least $f(k)$ has a hamiltonian anti-directed path? (If no such finite $f(k)$ exists then we set $f(k)=\infty$.)
	\end{openProblem}
Recall that any $2$-anti-traceable oriented graph is a tournament. Thus, it follows from Theorem \ref{tournament} that $f(2)=8$. In the next section we prove that $f(3)=3$ and $f(4)=8$. In Section \ref{almost hamiltonian},  using  Szemer\'{e}di's regularity lemma, we will show that every large $k$-anti-traceable oriented graph admits an anti-directed path that contains all but $o(n)$ vertices.

\section{Values of $f(3)$ and $f(4)$}\label{Sec:k=3and4}

We first prove that $f(3)=3$ and then that  $f(4)=8.$ 

	\subsection{$f(3)=3$}\label{Sec:k=3}
	To characterise $3$-traceable graphs we need the following definition.
	For a digraph $D$ and a set $X\subseteq V(D)$, $D[X]$ denotes the subdigraph of $D$ induced by $X.$
For a digraph $D$ with vertices $v_1,\dots ,v_n$ and a sequence of $n$ digraphs $H_1,\dots ,H_n$, the {\em composition} $\hat{D}=D[H_1,\dots ,H_n]$ has  $V(\hat{D})=\bigcup^n_{i=1}V(H_i)$ and  $$A(\hat{D})= \bigcup^n_{i=1}A(H_i)\cup \{x_ix_j:\ v_iv_j\in A(D), x_i\in V(H_i), x_j\in V(H_j), 1\le i\ne j\le n\}.$$
For two integers $p\le q$, let $[p,q]=\{i:\ p\le i\le q\}.$ If $p=1$, then $[q]=[p,q].$

	\begin{theorem}\label{prop:1}
		An oriented graph $D$ is 3-anti-traceable if and only if $D$ is an extended transitive tournament $T[I_1, I_2,\dots, I_t]$ for some positive integer $t$, where $I_i$ is an independent set with size at most two for all $i\in[t]$. 
	\end{theorem}
	\begin{proof}
		We first show the following two claims.
		
		\begin{claim}\label{cl:1}
			$D$ is transitive. 
		\end{claim}
		\begin{proof}
			If there are three vertices $v_1$, $v_2$ and $v_3$ in $D$ such that $v_1\to v_2\to v_3$, then one can see that $v_1\to v_3$ since otherwise $D[\{v_1,v_2,v_3\}]$ has no anti-directed path of length two. 
		\end{proof}
		\begin{claim}\label{cl:2}
			For any pair of non-adjacent vertices $u$ and $v$ in $D$, we have $V(D)\setminus\{u,v\}\subseteq (N^+(u)\cap N^+(v))\cup (N^-(u)\cap N^-(v))$. 
		\end{claim}
		\begin{proof}
			For any vertex $w\in V(D)\setminus \{u,v\}$, since $D[\{u,v,w\}]$ is anti-traceable, $w$ is either a common out- or in-neighbour of $u$ and $v$.
		\end{proof}
		We first prove the necessity. Let $u,v,w$ be any three different vertices of $D$. Suppose $u,v,w$ belong to three different independent sets. Then $D[\{u,v,w\}]$ is anti-traceable as $D[\{u,v,w\}]$ is a transitive tournament. 
		Suppose $u,v,w$ belong to two different independent sets. Without loss of generality, we assume $u$ and $v$ belong to the same independent set. Then $w$ is either a common out- or in-neighbour of $u$ and $v$. Therefore, 
		$D[\{u,v,w\}]$ is anti-traceable.
		
		Let $|V(D)|=n$. We now prove the sufficiency by induction on $n$. It is trivially true when $n\leq 2$. We can assume that $n\geq 3$. By Claim \ref{cl:1}, we have that $D$ is acyclic. Let $v_1$, $v_2$, $\dots$, $v_n$ be an acyclic ordering of $V(D)$. If $N^-(v_n)=V(D)\setminus\{v_n\}$, then by induction hypothesis $D-v_n$ is a required extended transitive tournament, which means $D$ is also a required one. If there exists $v_s$ with $1\leq s<n$ such that $v_s$ and $v_n$ are non-adjacent, then $s=n-1$, as otherwise $v_{n-1}$ should be either a common out- or in-neighbour of $v_s$ and $v_n$ by Claim \ref{cl:2}, a contradiction to the fact that $v_1$, $v_2$, $\dots$, $v_n$ is an acyclic ordering of $V(D)$. Thus, $v_{n-1}$ and $v_n$ are non-adjacent. Again, by Claim \ref{cl:2} and the fact that they are the last two vertices in the acyclic ordering, all vertices are common in-neighbours of $v_n$ and $v_{n-1}$. Thus, we are done by applying the induction hypothesis to $D-\{v_{n-1},v_n\}$. 
	\end{proof}
	The following results proved $f(3)=3$ in a slightly stronger form. One can observe from Theorem \ref{prop:1} that the following holds. 
	
	\begin{observation}\label{obs:order4}
		If $D$ is a 3-AT oriented graph with four vertices, then $D$ is anti-traceable. Moreover, $D$ has an hamiltonian anti-directed cycle unless $D\cong T_4$ (See Fig. \ref{fig:t4}). 
	\end{observation}
	\begin{proof}
		We can directly check that $D$ is anti-traceable when $D\cong T_4$. Let $D$ be 3-AT  with four vertices and $D\not\cong T_4$. By Theorem  \ref{prop:1}, one can observe that the last two vertices in the acyclic order are common out-neighbours of the first two vertices. Thus, $D$ is anti-traceable and has an hamiltonian anti-directed cycle. 
	\end{proof}
	
	\begin{figure}
		\begin{tikzpicture}[scale=1.7]
			\fill (0,1) circle (0.05);
			\fill (0,-1) circle (0.05);
			\fill (1,0) circle (0.05);
			\fill (-1,0) circle (0.05);
			\draw [thick] (-1,0)--(1,0);
			\draw [thick, arrows = {-Stealth[reversed, reversed]}] (0,1)--(1,0);
			\draw [thick, arrows = {-Stealth[reversed, reversed]}] (0,-1)--(1,0);
			\draw [thick, arrows = {-Stealth[reversed, reversed]}] (-1,0)--(0,0);
			\draw [thick, arrows = {-Stealth[reversed, reversed]}] (-1,0)--(0,1);
			\draw [thick, arrows = {-Stealth[reversed, reversed]}] (-1,-0)--(0,-1);
		\end{tikzpicture}
		\centering
		\caption{$T_4$}\label{fig:t4}
	\end{figure}
	
	\begin{theorem}
		If $D$ is 3-AT and $|V(D)|\geq 5$, then $D$ has an hamiltonian anti-directed cycle when $|V(D)|$ is even, and an anti-cycle with length $|V(D)|-1$ and an hamiltonian anti-directed path when $|V(D)|$ is odd. 
	\end{theorem}
	\begin{proof}
		We first prove the even case by induction on $|V(D)|$. The base case is $|V(D)|=6$. Let $u$ be a source of $D$ and $v$ be a sink. By Theorem \ref{prop:1}, $u\to v$ and $D-u-v$ is also 3-AT. Therefore, by Observation \ref{obs:order4}, $D-u-v$ has an hamiltonian anti-directed path $P=v_1v_2v_3v_4$, and we may assume without loss of generality that $v_1\to v_2\leftarrow v_3\to v_4$. By using Theorem \ref{prop:1} as well as the fact that $u$ is a source and $v$ is a sink, we have $u\to v_4$ and $v_1\to v$. Thus, $uvv_1v_2v_3v_4u$ is an hamiltonian anti-directed cycle in $D$. We now assume that $|V(D)|\geq 8$. By Propostion \ref{prop:1}, we see that $D=T[I_1, I_2,\dots, I_t]$ for some positive integer $t$. Let $u$ be a source and $v$ be a sink of $D$. Evidently, $u\to v$. By the induction hypothesis, $D-u-v$ has an hamiltonian anti-directed cycle $C$. Let $u'v'$ be an arc on $C$. Since $u'$ is not a sink, we get $u'\notin I_t$, which yields that $u'\to v$. Similarly, since $v'\notin I_1$, we have $u\to v'$. Thus, $uvu'Cv'u$ is an hamiltonian anti-directed cycle of $D$. This completes the proof of the even case.
		
		Now, we show the second statement. For $|V(D)|\geq 7$, we arbitrarily pick a vertex $v$ of $D$. Clearly, $|V(D-v)|\geq 6$ is even and $D-v$ is also 3-AT. Hence, $D-v$ has an hamiltonian anti-directed cycle $C$, which is a required cycle in $D$. For $|V(D)|=5$, we choose a vertex $v$, either a sink or a source, such that $D-v\not\cong T_4$ (such vertex must exist due to Theorem \ref{prop:1}). Thus, by Observation \ref{obs:order4}, $D-v$ has an hamiltonian anti-directed cycle which is also a required cycle.
		
		Now, add a new vertex $v^*$ and all arcs from $V(D)$ to $v^*$. Denote the new digraph by $D^*$. Note that $|V(D^*)|$ is even and $|V(D^*)|\geq 6$. By Theorem \ref{prop:1}, $D^*$ is 3-AT and thus contains an hamiltonian anti-directed cycle $C^*$. Therefore, $C^*-v^*$ is a required hamiltonian anti-directed path in $D$. 
	\end{proof}
	
	\subsection{$f(4)=8$}\label{Sec:k=4}

	\begin{proposition}
		$f(4)\geq 8$.
	\end{proposition}
	\begin{proof}
		As mentioned in Theorem \ref{tournament}, $PT_7$ does not contain an anti-directed path. Thus, we have $f(4)\geq 8$ since $PT_7$ is 4-AT.
	\end{proof}
	To make the following proofs easier to follow, we first give two simple observations.
	\begin{observation}\label{obser}
		Let $D$ be a 4-AT oriented graph and $a,b,c,d$  be four different vertices in $D$. Then
		\begin{description}
			\item[(i)] if $A(D[\{a,b,c,d\}])\setminus\{ad, da\}\subseteq \{ab, bc, ca, cd, db\}$, then $a\sim d$;
			
			\item[(ii)] if $D[\{a,b,c\}]$ is a subdigraph of a directed triangle and $a\not\sim d$, then $b\to d \leftarrow c$ or  $b \leftarrow d \to c$;
			\item[(iii)] if $A(D[\{a,b,c,d\}])\setminus \{ac,ca,bd,db\}$ forms a subdigraph of a directed $4$-cycle, then $a\sim c$ and $b\sim d$.
		\end{description}
	\end{observation}
	\begin{observation}\label{obser2}
		Let $D$ be a 4-AT oriented graph and $P=v_1v_2\dots v_s$ be a longest anti-directed path in $D$. Let $w$ be a vertex in $V(D)\setminus V(P)$. Then 
		\begin{description}
			\item[(i)] $\{w,v_2\}\not\subseteq N^+(v_1)$ and $\{w,v_2\}\not\subseteq N^-(v_1)$; $\{w,v_{s-1}\}\not\subseteq N^+(v_s)$ and\\ $\{w,v_{s-1}\}\not\subseteq N^-(v_s)$; 
			\item[(ii)]
			If $w\sim v_1$, then $D[\{v_1,v_2,w\}]$  is a  subdigraph of a directed triangle. Similarly, if $w\sim v_s$, then $D[\{v_{s-1},v_s,w\}]$ is a subdigraph of a directed triangle.
		\end{description}
	\end{observation}
	\begin{proof}
		(i). If $\{w,v_2\}\subseteq N^+(v_1)$ or $\{w,v_2\}\subseteq N^-(v_1)$, then $wv_1v_2\dots v_s$ is a longer anti-directed path, a contradiction.  If $\{w,v_{s-1}\}\subseteq N^+(v_s)$ or  $\{w,v_2\}\subseteq N^-(v_1)$, then $v_1v_2\dots v_sw$ is a longer anti-directed path, a contradiction.
		
		(ii) We assume without loss of generality that $v_1 \to v_2$. If $w\sim v_1$, then $w \to v_1$. If $w\to v_2$, then $v_1wv_2\dots v_s$ is a longer anti-directed path, a contradiction. Thus, we have $v_2\to w$ or $v_2\not\sim w$. Then, $D[\{v_1,v_2,w\}]$  is a subdigraph of a directed triangle. By a similar discussion, if $w\sim v_s$, then we also get  
		$D[\{v_{s-1},v_s,w\}]$ is a subdigraph of  a directed triangle. 
	\end{proof}
	\begin{lemma}\label{lem:v1simvs}
		Let $D$ be a 4-AT oriented graph and $P=v_1v_2\dots v_s$ be a longest anti-directed path in $D$. If $|V(P)|<|V(D)|$, then $v_1\sim v_s$.
	\end{lemma}
	\begin{proof}
		Let $w$ be a vertex in $V(D)\setminus V(P)$. Suppose to the contrary that $v_1\not\sim v_s$. Since $D[\{w,v_1,v_2,v_s\}]$ is anti-traceable, we have $w\sim v_1$ or $w\sim v_s$ (or both). We assume without loss of generality that $w\sim v_1$. We may further assume that $v_1\to v_2$ as the discussion for the case $v_2\to v_1$ is similar. By Observation \ref{obser2}, we have $w\to v_1$ and  $D[\{v_1,v_2,w\}]$ is a subdigraph of  directed triangle. 
		Since $v_1\not\sim v_s$, we $v_2\to v_s\leftarrow w$ or $ v_2\leftarrow v_s\to w$ by Observation \ref{obser}(ii). In particular, $w\sim v_s$ and therefore $D[\{v_{s-1},v_s,w\}]$ is a subdigraph of  directed triangle by Observation \ref{obser2}(ii). Thus, by Observation \ref{obser}(ii),
		we have $v_{s-1}\to v_1$ since $v_1\not\sim v_s$ and $w\to v_1$.
		
		Recall that $D[\{v_{s-1},v_s,w\}]$
		is a subdigraph of  a directed triangle and $v_2\to v_s\leftarrow w$ or $ v_2\leftarrow v_s\to w$. Thus, we have $v_s\to v_{s-1}$ if $v_2\to v_s\leftarrow w$ and  $v_{s-1}\to v_s$ if $ v_2\leftarrow v_s\to w$. 
		If 
		$ v_2 \to v_s\leftarrow w$, then  we  deduce that $v_1 \to v_2 \to v_s \to v_{s-1} \to v_1$. Since $v_1 \not\sim v_s$, by Observation \ref{obser}(iii), 
		$D[\{v_1,v_2,v_{s-1},v_s\}]$ cannot be anti-traceable, contradicting the fact that $D$ is 4-AT.  If $ v_2\leftarrow v_s\to w$, then  $wv_1v_{s-1}v_{s-2}\dots v_2v_s$ is a longer anti-directed path, a contradiction.
	\end{proof}
	
	\begin{lemma}\label{lem:4at1}
		Let $D$ be a 4-AT oriented graph and $P=v_1v_2\dots v_s$  be a longest anti-directed path in $D$. If $|V(P)|<|V(D)|$, then $|V(P)|$ is even.
	\end{lemma}
	\begin{proof}
		Suppose to the contrary that $s$ is odd. Let $w$ be a vertex in  $V(D)\setminus V(P)$.  By Lemma \ref{lem:v1simvs}, we have $v_1\sim v_s$.  Notice that when $s$ is odd, each vertex on  $C=v_1v_2v_3\dots v_sv_1$ is a source or a sink except for exactly one vertex. Thus, we may assume without loss of generality that $v_1\to v_2\leftarrow v_3 \dots v_{s-1}\leftarrow v_s \to v_1$.
		
		By Observation \ref{obser2}(i), we have $v_1\not \to w$ and $v_s\not \to w$.  If $w\to v_1$, then $wv_1v_sv_{s-1}\dots v_2$ is an anti-directed path longer than $P$, which leads to a contradiction. Thus, $w\not\sim v_1$. Since $D[\{w,v_1,v_2,v_s\}]$ is anti-traceable, we have $w\sim v_2$. By Observation \ref{obser}(iii) and the fact that $v_s\not\to v_1$ and $w\not\sim v_1$, we have  $w\to v_2$. But now $wv_2v_3\dots v_sv_1$ is a longer anti-directed path than $P$, a contradiction. 
	\end{proof}
	\begin{proposition}\label{prop:4to2kat}
		If $D$ is a 4-AT oriented graph, then  for any integer  $k\in[2, \frac{|V(D)|}{2}]$, we have $D$ is $2k$-AT. 
	\end{proposition}
	\begin{proof}
		We prove it by induction on $k$. When $k=2$, there is nothing left to prove. Suppose for any integer $\ell$ with $3\leq\ell<k$, $D$ is $2\ell$-AT. Now let $\ell=k$ and $S$ be an arbitrary subset of $V(D)$ with size $2k$. We only need to show that $D[S]$ is anti-traceable. 
		
		If $D[S]$ is a tournament, then we are done by Theorem \ref{tournament}. Thus, we may assume that $D[S]$ has two non-adjacent vertices $u$ and $v$. By the induction hypothesis, $D[S-u-v]$ is anti-traceable. Let $P=v_1v_2\dots v_{2k-2}$ be an hamiltonian anti-directed path in $D[S-u-v]$. If $P$ is not a longest path in $D[S]$, then by Lemma \ref{lem:4at1}, $D[S]$ is anti-traceable. 
		
		In the remaining proof, we always assume that $P$ is a longest anti-directed path in $D[S]$. By symmetry, we may further assume that $P=v_1\to v_2 \leftarrow v_3 \to \dots \leftarrow v_{2k-3}\to v_{2k-2}$. Since $D[\{u,v,v_{2k-3}, v_{2k-2}\}]$ is anti-traceable, one of $u$ and $v$ must be adjacent to $v_{2k-2}$. Assume without loss of generality that $v\sim v_{2k-2}$. By Observation \ref{obser2}, we have $v_{2k-2}\to v$ and $D[\{v,v_{2k-3},v_{2k-2}\}]$ is a subdigraph of a directed triangle. Since $u\not\sim v$ and $D[\{u,v,v_{2k-3},v_{2k-2}\}]$ is anti-traceable, by Observation \ref{obser}(ii), $u$ must be either a common out- or in-neighbour of $v_{2k-3}$ and $v_{2k-2}$. But now either $v_{2k-2}uv_{2k-3}\dots v_1$ or $uv_{2k-2}v_{2k-3}\dots v_1$  is a longer anti-directed path, which leads to a contradiction. 
	\end{proof}
	\begin{theorem}\label{main4}
		$f(4)=8$.
	\end{theorem}
	\begin{proof}
		Let $D$ be a 4-AT oriented graph with $|V(D)|=n\geq 8$. Let $P=v_1v_2v_3\dots v_s$ be a longest anti-directed path in $D$. Suppose to a contrary that $s<n$. By Lemma \ref{lem:4at1} and Proposition \ref{prop:4to2kat}, we have $s$ that is even and $s= n-1\geq 7$, which implies that $s\geq 8$. Let $w$ be the only vertex in $V(D)\setminus V(P)$. Since $s$ is even, we may assume by the symmetry that $v_1 \to v_2\leftarrow v_3 \cdots v_{s-1}\to v_s$. By Lemma \ref{lem:v1simvs}, $v_1$ and $v_s$ are adjacent. We consider the following two cases.
		
		{\bf Case 1.} $v_1\to v_s$. \\
		All subscripts are taken modulo $s$ in the following. Note that for any $1\leq i \leq s$, $v_i$ and $v_{i+1}$ are the endpoints of a longest anti-directed path. Thus, by Observation \ref{obser2}(i), we have $w\to v_i$ if $w\sim v_i$ and $i$ is odd, and $v_i\to w$ if $w\sim v_i$ and $i$ is even. Since $D[\{w,v_i,v_{i+1},v_{i+2}\}]$ is anti-traceable, by Observation \ref{obser}(i), we have $v_i\sim v_{i+2}$ for any $1\leq i \leq s$. Without loss of generality, we may assume $v_1\to v_3$ (otherwise we can switch the subscripts). Then, $w\to v_3$ for otherwise $D[\{w,v_1,v_2,v_3\}]$ is not anti-traceable (Observation \ref{obser}(iii)). One can deduce that $v_4\to v_2$ for otherwise $wv_3v_1v_sv_{s-1}\dots v_4v_2$ is an hamiltonian anti-directed path. Then, we have $v_4\to w$ for otherwise $D[\{w,v_2,v_3,v_4\}]$ is not anti-traceable (Observation \ref{obser}(iii)). Similar to the foregoing discussion, we get $w\to v_i$ and $v_i\to v_{i+2}$ when $i$ is odd, $v_i\to w$ and $v_i\leftarrow v_{i+2}$ when $i$ is even.
		
		Recall that $|V(P)|=s\geq 8$. If $D[\{V(P)\cup \{w\}\}]$ is a tournament, then by Theorem \ref{tournament}, $D[\{V(P)\cup \{w\}\}]$ is anti-traceable. Thus, without loss of generality, we assume $d_{D}(v_1)<n-1$. Let $k$ be the smallest subscript such that $v_1\not\sim v_k$. We have $4\leq k \leq s-2$. Notice that for any $1\leq i\leq s$, $D[\{w,v_i,v_{i+1}\}]$ is a directed triangle. Now consider 
		$D[\{w,v_1,v_{k-1},v_k\}]$, 
		$D[\{w,v_1,v_k,v_s\}]$ and $D[\{w,v_1,v_2,v_k\}]$, by Observation \ref{obser}(ii), We can determine the relationship between each of these three vertex pairs $\{v_{k-1}, v_1\}$, $\{v_k, v_s\}$ and  $\{v_k, v_2\}$. 
		When $v_k\to w$, we have $v_{k-1}\to v_1$, $v_k\to v_s$  and $v_k\to v_2$. Then $wv_1v_{k-1}v_{k-2}\dots v_2v_kv_sv_{s-1}\dots v_{k+1}$ is an hamiltonian anti-directed path, a contradiction. When $w\to v_k$, we have $v_{k-1}\to v_1$  and $v_2\to v_k$. Then $D[\{v_1,v_2,v_{k-1},v_k\}]$ is not anti-traceable (Observation \ref{obser}(iii)), a contradiction.

		{\bf Case 2.} $v_s\to v_1$. \\
		We first present the following two claims.
		\begin{claim}\label{wv1vs}
			$w\to v_1$ and $v_s\to w$.
		\end{claim}
		\begin{proof}
			By Observation \ref{obser2}(i), we have 
			$v_1\not\to w$ and $w\not\to v_s$. Assume to the contrary that $w\not\sim v_1$. 
			Note that $D[\{v_1,v_2,w,v_s\}]$ is anti-traceable. If $w\not\sim v_2$, then the anti-directed 
			hamiltonian path in $D[\{v_1,v_2,w,v_s\}]$ is $v_1\to v_2\leftarrow v_s \to w$. If $w\to v_2$, then consider another longest path $P'=wv_2v_3\dots v_s$. By Lemma \ref{lem:v1simvs}, $w\sim v_s$ and therefore $v_s\to w$ as $w\not\to v_s$. If $v_2\to w$, then the hamiltonian anti-directed path in $D[\{v_1,v_2,w,v_s\}]$  is $v_1\to v_2\leftarrow v_s \to w$ or $v_2\to w\leftarrow v_s \to v_1$. Thus, in all cases we have $v_s\to w$. Since $v_s\to w$, by Observation \ref{obser2}(ii), $D[\{w,v_{s-1},v_s\}]$  is a subdigraph of a directed triangle and then by Observation \ref{obser}(ii), we have $v_{s-1}\to v_1$.
			It follows that $wv_sv_1v_{s-1}v_{s-2}\dots v_2$ is a hamiltonian anti-directed path, a contradiction. Thus $w\to v_1$. 
			
			Suppose to the contrary that $w\not\sim v_s$. Since $w\to v_1$, by Observation \ref{obser2}(ii), we have $D[\{w,v_1,v_2\}]$  is a subdigraph of a directed triangle.  Then by Observation \ref{obser}(ii), we have $v_s\to v_2$. It follows that 
			$wv_1v_sv_2v_3\dots v_{s-1}$ is a longer anti-directed path when $v_2\to w$, a contradiction. Thus, $v_s\to w$.
		\end{proof}
		
		\begin{claim}\label{cl:vstov1}
			$v_2\to w$, $w\to v_{s-1}$, $v_1\to v_{s-1}$, $v_2\to v_s$, $w\to v_3$, $v_{s-2}\to w$, $v_1\to v_3$ and $v_{s-2}\to v_s$.
		\end{claim}
		\begin{proof}
			By Claim \ref{wv1vs}, if $v_s\to v_2$ or $v_{s-1}\to v_1$, then $wv_1v_sv_2v_3\dots v_{s-1}$ or $wv_sv_1v_{s-1}v_{s-2}\dots v_2$ is an hamiltonian anti-directed path. Thus, we have $v_2\to v_s$  when $v_2\sim v_s$ and $v_1\to v_{s-1}$ when $v_1\sim v_{s-1}$. By  Observation \ref{obser2}(ii), we have $v_2\to w$  when $v_2\sim w$ and $w\to v_{s-1}$ when $w\sim v_{s-1}$. Considering  $D[\{w,v_1,v_2,v_s\}]$, since $D[\{v_1,v_2,v_s\}]$ is a subdigraph of a directed triangle and $v_s\to w\to v_1$, by Observation \ref{obser}(ii), we have $v_2\sim w$ and therefore $v_2\to w$. Similarly, considering  $D[\{w,v_1,v_{s-1},v_s\}]$, we get $w\to v_{s-1}$. 
			
			Next, we will show that $v_2\to v_s$ and $v_1\to v_{s-1}$. Suppose $v_2\not\sim v_{s-1}$ or $v_2\to v_{s-1}$. Since $D[\{v_1,v_2,v_{s-1},v_s\}]$ is anti-traceable, by Observation \ref{obser}(iii), we have  $v_2\sim v_s$ and $v_1\sim v_{s-1}$ and therefore $v_2\to v_s$ and $v_1\to v_{s-1}$. Suppose $v_{s-1}\to v_2$. Considering  $D[\{w,v_1,v_2,v_{s-1}\}]$, since $D[\{w,v_1,v_2\}]$ is a directed triangle and $w\to v_{s-1}\to v_2$, by Observation \ref{obser}(ii), we have $v_1\sim v_{s-1}$ and so $v_1\to v_{s-1}$. Similarly, considering  $D[\{w,v_2,v_{s-1},v_s\}]$, we get $v_2\to v_s$. 
			
			Since both $wv_2v_sv_{s-1}\dots v_3$ and $wv_{s-1}v_1v_2\dots v_{s-2}$ are longest anti-directed paths, we have $w\to v_3$ and $v_{s-2}\to w$ as otherwise we may consider these two paths and are done by {\bf Case 1}. 
			As  $D[\{w,v_1,v_2,v_3\}]$ and $D[\{w,v_{s-2},v_{s-1},v_s\}]$ are anti-traceable, by Observation \ref{obser}(i), we have $v_1\sim v_3$ and $v_{s-2}\sim v_s$.
			If $v_3\to v_1$ (or $v_s \to v_{s-2}$), then $wv_2v_sv_{s-1}\dots v_3v_1$ (or $wv_{s-1}v_1v_2\dots v_{s-2}v_s$) is a longer anti-directed path than $P$, a contradiction. We thus deduce that $v_1\to v_3$ and $v_{s-2}\to v_s$.
		\end{proof}
		
		Note that  Claim \ref{wv1vs} and  Claim \ref{cl:vstov1} are true for any anti-directed path $Q=u_1u_2\dots u_s$ of $D$ satisfying that $|Q|=|P|$, $u_1\to u_2$ and $u_s\to u_1$. 
		
		\begin{claim}\label{cl:end}
			For any anti-directed path $Q=u_1u_2\dots u_s$ of $D$ satisfying that $|Q|=|P|$, $u_1\to u_2$ and $u_s\to u_1$, we have $u_{s-2(t+1)}\to u_{s-2t}$ for all $0\leq t\leq \frac{s}{2}-2$.
		\end{claim}
		\begin{proof}
			We prove this claim by induction on integer $t$. By Claim \ref{cl:vstov1}, it is true for $t=0$. Assume that $t\geq 1$ it and is true for all integers less than $t$. Now we prove it is true for $t$. Let  $Q=u_1u_2\dots u_s$ be an anti-directed path of $D$ satisfying that $|Q|=|P|$, $u_1\to u_2$ and $u_s\to u_1$. We need to prove $u_{s-2t-2}\to u_{s-2t}$. Let  $v$ be the vertex in  $V(D)\setminus V(Q)$.  By Claim \ref{wv1vs} and  Claim \ref{cl:vstov1}, we have $Q'=u_1'u_2'\dots u_s'=vu_{s-1}u_1u_2u_3\dots u_{s-4}u_{s-3}u_{s-2}$ is an anti-directed path satisfying that $|Q'|=|P|$, $u_1'\to u_2'$ and  $u_s' \to u_1'$. Then for $Q'$, by the induction hypothesis, we have $u_{s-2t}'\to u_{s-2(t-1)}'$, which means that $u_{s-2t-2}\to u_{s-2t}$ and thus we are done.
		\end{proof}
		
		Let $P'=v_1'v_2'\dots v_s'=wv_{s-1}v_1v_2v_3\dots v_{s-4}v_{s-3}v_{s-2}$. Note that $P'$ is an anti-directed path satisfying that $v_1'\to v_2'$, $v_s' \to v_1'$ and $|P'|=|P|$. For $P'$, by Claim \ref{cl:end}, we get $v'_2\to v'_4$ and $v'_6\to v'_8$, which implies $v_{s-1}\to v_2$ and $v_4\to v_6$. Since $v_{s-1}\to v_2$, we get an anti-directed path  $P''=v_1''v_2''\dots v_s''=wv_3v_1v_2v_{s-1}v_{s-2}\dots v_6v_5v_4$. Note that $P''$ satisfies $|P''|=|P|$ and $v_1''\to v_2''$ as $w\to v_3$. Due to {\bf Case 1}, we may assume $v''_s\to v''_1$. Thus, by Claim \ref{cl:vstov1}, we have $v_{s-2}''\to v_{s}''$. This implies that $v_6\to v_4$, contradicting the fact that $D$ is an oriented graph.
		
		This completes the proof of Theorem \ref{main4}.
	\end{proof}
	
	\section{Almost spanning anti-directed paths}\label{almost hamiltonian}
	In this section we show that for every $\varepsilon >0$ every sufficiently large dense  oriented graph $G$ has an anti-directed path with  at least $(1-\varepsilon)|V(G)|$ vertices. Note that every oriented graph on $n$ vertices that is $k$-antitracable has minimum degree at least $n-k+1$, since if there were a vertex $v$ of degree less than $n-k+1$ then  $v$ and $k-1$ non-neighbours of $v$ would form a set of size $k$ without an anti-directed path. To prove the result we will use the celebrated regularity lemma. To do so we need to introduce the following notation and well-known results.

	\begin{definition}  Let $V$ be the vertex set of a digraph $G$, and let $X, Y$ be disjoint subsets of $V$. Furthermore, let $(X, Y)_G$ be the set of arcs that start in $X$ and end in $Y$. Then the \emph{arc density} of $X$ and $Y$ is
		\[d(X, Y) = \frac{|(X, Y)_G|}{|X||Y|}.\]
	\end{definition}

	\begin{definition} Let $G$ be a digraph with vertex set $V$ and let $\varepsilon > 0$. An ordered pair $(X, Y)$ of disjoint subsets  of $V$ is $\varepsilon$-\emph{regular} if for all subsets $A \subseteq X, B \subseteq Y$ with $|A| \geq \varepsilon|X|, |B| \geq \varepsilon|Y|$,
		\[|d(X,Y) - d(A,B)| \leq \varepsilon.\] 
	\end{definition}

	\begin{definition} 
		Let $G$ be a digraph with vertex set $V$ and let $\mathcal{P} = \{V_0, V_1, ..., V_k\}$ be a partition of $V$ into $k+1$ sets. Then, $\mathcal{P}$ is called an $\varepsilon$-\emph{regular partition} of $G$  if it satisfies the following conditions:
		
		\begin{itemize}
			\item $|V_0| \leq \varepsilon|V|$,
			\item  $|V_1| = ... = |V_k|$, 
			\item  all but at most $\varepsilon k^2$ of ordered pairs $(V_i, V_j)$ with $1 \leq i < j \leq k$ are $\varepsilon$-regular.
		\end{itemize}
	\end{definition}

	We are now able to state Szemer\'{e}di's regularly lemma for oriented graphs which was proved by Alon and Shapira \cite{alon}.
	\begin{theorem} \label{thm:szem} If $\varepsilon \geq 0$ and $m \in \mathbb{N}$, then there exists $M \in \mathbb{N}$ such that if $G$ is an oriented graph of order at least $M$ then there exists an $\varepsilon$-regular partition of $G$ with $k$ parts $\{V_1, V_2, ..., V_k \}$, where $m \leq k \leq M$. 
	\end{theorem}
	
	It is well known that $\varepsilon$-regularity of a  pair $(X,Y)$ implies conditions on the degrees of vertices. For example, by considering  the set $X'$ in $X$ of all vertices of out-degree less than $(d(X,Y)-\varepsilon)|Y'|$ into a set $Y'\subseteq Y$ with $|Y'| \geq \varepsilon|Y|$,  it follows immediately from the definition that $|X'|\leq \varepsilon |X|$. We state this observation as a lemma for easier reference. 
	
	\begin{lemma}  \label{lem:degree}  Let $(X, Y)$ be a $\varepsilon$-regular pair with density $d$ and let $Y' \subseteq Y~(X'\subseteq X, resp.)$ be of size at least $\varepsilon|Y|~(\varepsilon|X|, resp. )$. Then for all but at most $\varepsilon|X|~(\varepsilon|Y|,resp.)$  vertices $v \in X~(v\in Y, resp.)$, the inequality $|N_{Y'}^+ (v)| \geq (d - \varepsilon)|Y'|~(|N_{X'}^- (v)| \geq (d - \varepsilon)|X'|,resp.)$ holds.
	\end{lemma}
	
	The following lemma shows that one can find long anti-directed paths in an $\varepsilon$-regular pair.
	
	\begin{lemma} \label{lem:long_path_in_regular_pairs}
		Let $ \varepsilon>0$ and let $d\geq 5\varepsilon$. Let $(X,Y)$ be an $\varepsilon$-regular pair of density at least $d$ with $|X|=|Y|= n$, and let $X'\subset X$ and $Y'\subset Y$ satisfy $|X'|,|Y'| <\varepsilon n$. Then for all $x\in X\setminus X'$ and all $y\in Y\setminus Y'$ of degree at least $(d-\varepsilon)n$ there exists and anti-directed path in $X\setminus X' \cup Y\setminus Y'$ starting at $x$ and ending in $y$ of length at least $(1-\frac{\varepsilon}{d-\varepsilon} - 3 \varepsilon) (|X|+|Y|)$. 
		
	\end{lemma}
	\begin{proof}
		We first fix a set $X^*\subset X \setminus (X' \cup \{x\})$ of size $\varepsilon|X|$ in the neighbourhood of $y$. Note that such a set exists as $y$ has degree at least $(d-\varepsilon)n- |X'|-1  \geq  \varepsilon n $ in  $X \setminus (X' \cup \{x\})$. We now build the path iteratively. To do so we set $\bar{X}_0= X' \cup X^* \cup \{x\}$, $\bar{Y}_0= Y' \cup \{y\}$, and $p_0=x$. At step $t$ we assume that we have built a path from $x$ to $p_t$ using only vertices in $\bar{X}_t\cup \bar{Y}_t$ and that $p_t$ has in-degree more than  $\varepsilon n $ into $X\setminus \bar{X}_t$ if $p_t\in Y$ and out-degree more than  $\varepsilon n$ into $Y\setminus \bar{Y}_t$ if $p_t\in X$. Note that these conditions are satisfied if $t=0$ as the degree of $x$ into $Y\setminus \bar{Y}_0$ is at least $(d-\varepsilon) n - |\bar{Y}_0| > \varepsilon n$. We stop the process when $X\setminus \bar{X}_t$ is smaller than  $\varepsilon n/(d-\varepsilon)+2$ .(It  will turn out  that  $|X\setminus \bar{X}_t|$ is always smaller than  or equal to $|Y\setminus \bar{Y}_t|$ so this is a bound on both sets). 
		
		At step $t+1$, we consider $p_t$. First assume that $p_t\in X$. By Lemma~\ref{lem:degree}, there are at most $\varepsilon n$ vertices in $Y$ with degree less than $(d-\varepsilon)| X \setminus \bar{X}_t|$ into $X\setminus \bar{X}_t$. As we have not stopped the process, $(d-\varepsilon)|X \setminus \bar{X}_t|>\varepsilon n$. As the degree of $p_t$ is bigger than $\varepsilon n$ we can choose $p_{t+1}$ as a neighbour of $p_t$ that has $\varepsilon n$ neighbours in $X\setminus \bar{X}_t$. We then add $p_{t+1}$ to $\bar{Y}_t$ to obtain $\bar{Y}_{t+1}$ and set  $\bar{X}_{t+1}:= \bar{X}_t$.  If  $p_t\in Y$ then we proceed analogously: By Lemma~\ref{lem:degree} there are at most $\varepsilon n$ vertices in $X$ with degree less than $(d-\varepsilon)| Y \setminus \bar{Y}_t|$ into $Y\setminus \bar{Y}_t$. As we have not stopped the process $(d-\varepsilon)|Y  \setminus \bar{Y}_t|>\varepsilon n$. As the degree of $p_t$ is bigger than $\varepsilon n$ we can choose $p_{t+1}$ as a neighbour of $p_t$ that has $\varepsilon n$ neighbours in $Y\setminus \bar{Y}_t$. We then add $p_{t+1}$ to $\bar{X}_t$ to obtain $\bar{X}_{t+1}$ and set  $\bar{Y}_{t+1}:= \bar{Y}_t$.  
		When this process has stopped at vertex $x^* \in X$  we have an anti-directed path of length at least  $(1-\frac{\varepsilon}{d-\varepsilon} - 3 \varepsilon )(|X|+|Y|)$ from $x$ to $x^*$   as we started with $X\setminus \bar{X}_0$ of size at least $n- 2\varepsilon n  -1$  and reduced the size of this set by $1$ at every second step while adding a vertex to the path at every step until we reached a set of size $\varepsilon n/(d-\varepsilon)+1$. We continue this path to $y$ by choosing a neighbour of $x^*$  that is not part of the anti-directed path that is a neighbour of a vertex in $X^*$. As the neighbourhood of $x^*$ that is not part of the anti-directed path is of size at least $\varepsilon n$ and $X^*$ is of size $\varepsilon n$ the density of this set is at least $d-\varepsilon$ and thus there are many possibilities.
	\end{proof}
	

	\begin{theorem} Let  $\varepsilon >0$,  $n,k \in \mathbb{N}$ . For every sufficiently large $n$ and every $k=o(n)$, every oriented graph $D$ on $n$ vertices with minimum degree $n-k$, has an anti-directed path of length at least $(1-\varepsilon)n$.
	\end{theorem}
	\begin{proof} 
		Let $\varepsilon' < \varepsilon/10$. We apply Theorem~\ref{thm:szem} with $\varepsilon'$ (and $m=1$) to obtain an $\varepsilon'$-regular partition $V_0,\ldots, V_\ell$. We consider the undirected simple reduced graph $R$ which has a vertex for each $V_1, \ldots, V_\ell$ and an edge between $V_i, V_j$ if $(V_i,V_j)$  or $(V_j,V_i)$ form an $\varepsilon'$-regular pair of density at least $1/2-\varepsilon'$. Note that there are at least $|V_i|(|V_j|-k)$ arcs between $V_i$ and $V_j$ in either direction and it follows that for sufficiently large $n$ at least one of the pairs $(V_i,V_j)$ of $(V_j,V_i)$ have density at least $d=1/2-\varepsilon'$. Thus $R$ is a complete graph with at most $\varepsilon' l^2$ edges missing. 
		It is well known that a complete graph can be decomposed into (nearly) perfect matchings (in fact, a complete graph with an even number of vertices can be decomposed into perfect matchings and an old result attributed to Walecki by  Lucas in 1892 \cite{lucas} says that every complete graph with an odd number of vertices can be decomposed into hamiltonian cycles).
		It follows that $R$ must have one matching with at least $\frac{{\ell \choose 2}-\varepsilon' \ell^2 }{\ell}= \frac{(1-2\varepsilon')\ell-1}{2}$ many edges. Consider one of these matchings. Each edge of this matching corresponds to an $\varepsilon'$-regular (directed) pair $(V_i,V_j)$ of density at least $d$. Let $(A_1,B_1),\ldots,(A_t,B_t)$ be these $\varepsilon'$-regular pairs of density at least $d$. 
		
		We know from Lemma~\ref{lem:long_path_in_regular_pairs} that $(A_i,B_i)$ contain long anti-directed paths and we now choose the starting and end vertex for each pair in a way that we can connect the anti-paths in $(A_i,B_i)$ and $(A_{i+1},B_{i+1})$ to get a almost hamiltonian anti-path in $D$.
		For $a_1\in A_1$ we choose any vertex with out-degree into $B_1$ at least $(d-\varepsilon') |V(B_1)|$. By Lemma~\ref{lem:degree} many such vertices exist. For $i=1,\ldots,t-1$, we consider two cases: Either there exists an arc $(v,u)$ between a vertex $u$ of $B_i$ of in-degree at least  $(d-\varepsilon') |V(A_i)|$ from $A_i$  and a vertex $v$ of  $A_{i+1}$ of out-degree at least  $(d-\varepsilon') |V(B_{i+1})|$ into $B_{i+1}$, in which case we set $a_{i+1}=v$ and $b_i=u$; or all arcs between vertices of sufficiently large degree are directed from $B_i$ to $A_{i+1}$. In the latter case, we choose an arc $(u,v)$ in $B_i$ between vertices of in-degree at least $(d-\varepsilon') |V(A_i)|$ from $A_i$. And then we choose an arc $(x,y)$ within the subgraph induced by $u$'s out-neighbours in vertices in $A_{i+1}$ with out-degree at least $(d-\varepsilon')|V(B_{i+1})|$ into $B_{i+1}$. Note that such arcs exist as by Lemma~\ref{lem:degree} the set of these vertices is much larger than $k$ and the minimum degree is $n-k$. In addition, the path $v\leftarrow u \to x \leftarrow y$ can connect any $(a_i,v)$ anti-path in $(A_i,B_i)$ obtained by Lemma \ref{lem:long_path_in_regular_pairs} and any anti-path in $(A_i,B_i)$ starts at $y$. Thus,
		we choose $b_i=v$, $a_{i+1}=x$ and set $B_i'=\{u\}$ and  $A'_{i+1}=\{y\}$ when applying Lemma \ref{lem:long_path_in_regular_pairs}. Finally we choose $b_t$ any vertex with in-degree into $A_t$ at least  $(d-\varepsilon') |V(A_t)|$. Again by Lemma~\ref{lem:degree} many such vertices exist. 
		
		Therefore, we formed an anti-path with order at least
		\[\frac{(1-2\varepsilon')\ell-1}{2}\times(1-\frac{\varepsilon'}{d -  \varepsilon'}-3\varepsilon')(\frac{2(1-\varepsilon')n}{\ell})\geq(1-9\varepsilon')(1-\varepsilon')n>(1-\varepsilon)n,\]
		which completes the proof.
	\end{proof}

	\maketitle

\end{document}